\documentclass[12pt]{amsart}

\usepackage[utf8]{inputenc} 
\usepackage[T1]{fontenc} 
\usepackage[english]{babel} 
\usepackage{mathrsfs} 
\usepackage{amssymb}
\usepackage[margin=1.0in]{geometry}

\theoremstyle{remark} 
\newtheorem*{remark}{Remark}

\theoremstyle{plain} 
\newtheorem{thm}{Theorem} 
\newtheorem{lem}[thm]{Lemma} 

\newtheorem{prop}[thm]{Proposition}

\DeclareMathOperator{\Id}{id}

\newcommand{\N}{\mathbb{N}}
\newcommand{\M}{\mathcal{M}}
\newcommand{\Mo}{\mathcal{M}_0}
\begin{document} 
\title{On the spaces of bounded and compact multiplicative Hankel operators}
\date{\today} 

\author{Karl-Mikael Perfekt}
\address{Department of Mathematics and Statistics,
	University of Reading, Reading RG6 6AX, United Kingdom}
\email{k.perfekt@reading.ac.uk}

\begin{abstract}
	A multiplicative Hankel operator is an operator with matrix representation $M(\alpha) = \{\alpha(nm)\}_{n,m=1}^\infty$, where $\alpha$ is the generating sequence of $M(\alpha)$. Let $\mathcal{M}$ and $\mathcal{M}_0$ denote the spaces of bounded and compact multiplicative Hankel operators, respectively. In this note it is shown that the distance from an operator $M(\alpha) \in \mathcal{M}$ to the compact operators is minimized by a nonunique compact multiplicative Hankel operator $N(\beta) \in \mathcal{M}_0$,
	$$\|M(\alpha) - N(\beta)\|_{\mathcal{B}(\ell^2(\N))} = \inf \left \{\|M(\alpha) - K \|_{\mathcal{B}(\ell^2(\N))} \, : \, K \colon \ell^2(\N) \to \ell^2(\N) \textrm{ compact} \right\}.$$
	Intimately connected with this result, it is then proven that the bidual of $\mathcal{M}_0$ is isometrically isomorphic to $\mathcal{M}$, $\mathcal{M}_0^{\ast \ast} \simeq \mathcal{M}$. It follows that $\Mo$ is an M-ideal in $\M$. The dual space $\Mo^\ast$ is isometrically isomorphic to a projective tensor product with respect to Dirichlet convolution. The stated results are also valid for small Hankel operators on the Hardy space $H^2(\mathbb{D}^d)$ of a finite polydisk.
	
\end{abstract}

\subjclass[2010]{}

\maketitle
\section{Introduction}
Given a sequence $\alpha \colon \N \to \mathbb{C}$, we consider the corresponding multiplicative Hankel operator $m = M(\alpha) \colon \ell^2(\N) \to \ell^2(\N)$, defined by 
$$\langle M(\alpha) a, b \rangle_{\ell^2(\N)} = \sum_{n,m=1}^\infty a(n) \overline{b(m)} \alpha(nm), \quad a, b \in \ell^2(\N).$$
Initially, we consider this equality only for finite sequences $a$ and $b$. It defines a bounded operator $M(\alpha) \colon \ell^2(\N) \to \ell^2(\N)$, with matrix representation $\{\alpha(nm)\}_{n,m=1}^\infty$ in the standard basis of $\ell^2(\N)$, if and only if there is a constant $C > 0$ such that
$$\left| \langle M(\alpha) a, b \rangle_{\ell^2(\N)} \right| \leq C\|a\|_{\ell^2(\N)}\|b\|_{\ell^2(\N)}, \quad a,b \textrm{ finite sequences.}$$
Multiplicative Hankel operators are also known as Helson matrices, having been introduced by Helson in \cite{Helson06, Helson10}. 

There are two common alternative interpretations. One is in terms of Dirichlet series. Let $\mathcal{H}^2$ be the Hardy space of Dirichlet series, the Hilbert space with $(n^{-s})_{n=1}^\infty$ as a basis. If 
$$f(s) = \sum_{n=1}^\infty a(n) n^{-s}, \; g(s) = \sum_{n=1}^\infty \overline{b(n)} n^{-s}, \; \rho(s) = \sum_{n=1}^\infty \overline{\alpha(n)} n^{-s},$$
 then
$$\langle M(\alpha) a, b \rangle_{\ell^2(\N)} = \langle fg, \rho \rangle_{\mathcal{H}^2}.$$
Hence there is an isometric correspondence between Helson matrices and Hankel operators on $\mathcal{H}^2$, since the forms associated the latter are precisely of the type $(f,g) \mapsto \langle fg, \rho \rangle_{\mathcal{H}^2}$.

The second interpretation is in terms of the Hardy space of the infinite polytorus $H^2(\mathbb{T}^\infty)$, the Hilbert space with basis $(z^\kappa)_{\kappa}$, where $z = (z_1, z_2, \ldots)$, and $\kappa = (\kappa_1, \kappa_2, \ldots)$ runs through the countably infinite, but finitely supported, multi-indices.  Identify each integer $n$ with a multi-index $\kappa$ of this type through the factorization of $n$ into the primes $p_1, p_2, \ldots$,
$$n \longleftrightarrow \kappa \; \textrm{  if and only if  } \; n = \prod_{j=1}^\infty p_j^{\kappa_j}.$$
Under this equivalence, multiplicative Hankel operators correspond to additive Hankel operators on a countably infinite number of variables,
$$\langle M(\alpha) a, b \rangle_{\ell^2(\N)} = \sum_{\kappa, \kappa'} a(\kappa) \overline{b(\kappa')} \alpha(\kappa + \kappa').$$
Hence the multiplicative Hankel operators correspond isometrically to small Hankel operators on $H^2(\mathbb{T}^\infty)$, since the matrix representations of the latter are of the form $\{\alpha(\kappa + \kappa')\}_{\kappa, \kappa'}$. See \cite{Helson06, Helson10} for details. 

In particular, the Helson matrices generalize the small Hankel operators on the Hardy space of any finite polytorus $H^2(\mathbb{T}^d)$, $d < \infty$. In fact, the results in this note have analogous statements for small Hankel operators on $H^2(\mathbb{T}^d)$; every proof given remains valid verbatim after restricting the number of prime factors, that is, the number of variables.

The first result is the following. We denote by $\mathcal{B}(\ell^2(\N))$ and $\mathcal{K}(\ell^2(\N))$, respectively, the spaces of bounded and compact operators on $\ell^2(\N)$.
\begin{thm} \label{thm:distcpct}
Let $M(\alpha)$ be a bounded multiplicative Hankel operator. Then there exists a compact multiplicative Hankel operator $N(\beta)$ such that
\begin{equation} \label{eq:essnorm}
\|M(\alpha) - N(\beta)\|_{\mathcal{B}(\ell^2(\N))} = \inf \left \{\|M(\alpha) - K \|_{\mathcal{B}(\ell^2(\N))} \, : \, K \in \mathcal{K}(\ell^2(\N)) \right\}.
\end{equation}
The minimizer $N(\beta)$ is never unique, unless $M(\alpha)$ is compact.
\end{thm}
The quantity on the right-hand side of \eqref{eq:essnorm} is known as the essential norm of $M(\alpha)$. For classical Hankel operators on $H^2(\mathbb{T})$, this result was proven by Axler, Berg, Jewell, and Shields in \cite{ABJS79}, and can be viewed as a limiting case of the theory of Adamjan, Arov, and Krein \cite{AAK71}. The demonstration of Theorem~\ref{thm:distcpct} requires only a minor modification of the arguments in \cite{ABJS79}, the main point being that a characterization of the class of bounded multiplicative Hankel operators is not necessary for the proof.
 
On $H^2(\mathbb{T})$, Nehari's theorem \cite{Nehari} states that the class of bounded Hankel operators can be isometrically identified with $L^\infty(\mathbb{T})/H^{\infty}(\mathbb{T})$, where $L^\infty(\mathbb{T})$ and $H^\infty(\mathbb{T})$ denote the spaces of bounded and bounded analytic functions on $\mathbb{T}$, respectively. By Hartman's theorem \cite{Hartman}, the class of compact Hankel operators is isometrically isomorphic to $(H^\infty(\mathbb{T}) + C(\mathbb{T}))/H^{\infty}(\mathbb{T})$, where $C(\mathbb{T})$ denotes the space of continuous functions on $\mathbb{T}$. Note that the spaces $L^\infty$, $H^\infty$, and $H^\infty + C$ are all algebras, as proven by Sarason \cite{Sara73}.

Luecking \cite{Lu80} observed, through a very illustrative argument relying on function algebra techniques, that the compact Hankel operators form an M-ideal in the space of bounded Hankel operators. M-ideality implies proximinality; the distance from a bounded Hankel operator to the compact Hankel operators has a minimizer. Thus Luecking reproved some of the results of \cite{ABJS79}. Since
$$\left( (H^{\infty} + C)/H^\infty \right)^{\ast \ast} \simeq L^\infty/H^\infty,$$
it follows that the bidual of the space of compact Hankel operators is isometrically isomorphic to the space of bounded Hankel operators. Spaces which are M-ideals in their biduals are said to be M-embedded.

The multiplicative Hankel operators, on the other hand, have thus far resisted all attempts to characterize their boundedness. It has been shown that a Nehari-type theorem cannot exist \cite{OS}, and positive results only exist in special cases \cite{Helson06, PP17}. In spite of this, the main theorem shows that Luecking's result holds for multiplicative Hankel operators.

Let 
$$\Mo = \{m = M(\alpha) \, : \, M(\alpha) \colon \ell^2(\N) \to \ell^2(\N) \textrm{ compact} \}$$
and
$$\M = \{m = M(\alpha) \, : \, M(\alpha) \colon \ell^2(\N) \to \ell^2(\N) \textrm{ bounded} \}.$$
Equipped with the operator norm, $\Mo$ and $\M$ are closed subspaces of $\mathcal{K}(\ell^2(\N))$ and $\mathcal{B}(\ell^2(\N))$, respectively.
For a Banach space $Y$, we denote by $\iota_{Y}$ the canonical embedding $\iota_{Y} \colon Y \to Y^{\ast \ast}$,
$$\iota_{Y} y (y^\ast) = y^\ast(y), \quad y \in Y, \; y^\ast \in Y^\ast.$$
\begin{thm} \label{thm:mideal}
There is a unique isometric isomorphism $U \colon \Mo^{\ast \ast} \to \M$ such that $U \iota_{\Mo} m = m$ for every $m \in \Mo$. Furthermore, $\Mo$ is an M-ideal in $\M$. 
\end{thm}
\begin{remark}
As pointed out earlier, Theorem~\ref{thm:mideal} is also true when stated for small Hankel operators on $H^2(\mathbb{T}^d)$, $d<\infty$. The biduality has in this case been observed isomorphically in \cite{LTW06}, with an argument based on the non-isometric Nehari-type theorems proven in \cite{FL02, LT09}. 
\end{remark}
The M-ideal property means the following: there is an (onto) projection $L \colon \M^\ast \to \Mo^\perp$ such that
$$\|m^\ast\|_{\M^\ast} = \|Lm^\ast\|_{\M^\ast} + \|m^\ast - Lm^\ast\|_{\M^\ast}, \quad m^\ast \in \M^\ast,$$
where $\Mo^\perp$ denotes the space of functionals $m^\ast \in \M^\ast$ which annihilate $\Mo$. M-ideals were introduced by Alfsen and Effros \cite{AE72} as a Banach space analogue of closed two-sided ideals in $C^\ast$-algebras. Very loosely speaking, the fact that $\Mo$ is an M-ideal in $\M$ implies that the norm of $\M$ resembles a maximum norm and, in this analogy, that $\Mo$ is the subspace of elements vanishing at infinity. The book \cite{HWW} comprehensively treats M-structure theory and its applications.

We will make use of the following consequences of Theorem~\ref{thm:mideal}. Proximinality of $\Mo$ in $\M$ was already mentioned, but the M-ideal property also implies that the minimizer is never unique \cite{HSW75}. It also ensures that $\Mo^\ast$ is a strongly unique predual of $\M$ \cite[Proposition III.2.10]{HWW}. This means that every isometric isomorphism of $\M$ onto $Y^\ast$, $Y$ a Banach space, is weak$^\ast$-weak$^\ast$ continuous, that is, arises as the adjoint of an isometric isomorphism of $Y$ onto $\Mo^\ast$. On the other hand, $\Mo^\ast$ has infinitely many different preduals \cite[Theoreme 27]{godefroy}.

The predual of $\M$ is well known to have an almost tautological characterization as a projective tensor product with respect to Dirichlet convolution,
$$\mathcal{X} = \ell^2(\N) \; \hat{\star} \; \ell^2(\N).$$
The space $\mathcal{X}$ is also referred to as a weak product space. We defer the precise definition to the next section -- after establishing the main theorems, we essentially show, following \cite{RS14}, that all reasonable definitions of $\mathcal{X}$ coincide. 
\begin{thm}\label{thm:midspace}
There is an isometric isomorphism $L \colon \mathcal{X} \to \Mo^\ast$ such that $L^\ast U^{-1} \colon \M \to \mathcal{X}^\ast$ is the canonical isometric isomorphism of $\M$ onto $\mathcal{X}^\ast$, where $U \colon \Mo^{\ast \ast} \to \M$ is the isometric isomorphism of Theorem~\ref{thm:mideal}.
\end{thm}
Informally stated, $\Mo^\ast \simeq \mathcal{X}$ and $\mathcal{X}^\ast \simeq \M$.  Theorem~\ref{thm:midspace} follows at once from Theorem~\ref{thm:mideal} and the uniqueness of the predual of $\M$, but we also supply a direct proof.
While the duality $\mathcal{X}^\ast \simeq \M$ is a rephrasing of the definition of $\M$, it is difficult to identify a common approach to dualities of the type $\Mo^\ast \simeq \mathcal{X}$ in the existing literature. Often, the latter duality is deduced (isomorphically) via a concrete description of $\M$. For a small selection of relevant examples, see \cite{AFP85, CW77, HWW, LTW06, Li13, Per13, Wu93}.

The idea behind this note is that the direct view of $\M$ as a subspace of $\mathcal{B}(\ell^2(\N))$  already provides sufficient information to prove Theorems~\ref{thm:distcpct}, \ref{thm:mideal}, and \ref{thm:midspace}. In this direction, Wu \cite{Wu93} worked with an embedding into the space of bounded operators to deduce duality results for certain Hankel-type forms on Dirichlet spaces.

The proofs of the results only have two main ingredients. The first is a device to approximate elements of $\M$ by elements of $\Mo$ (Lemma~\ref{lem:dilation}). Such an approximation property is necessary, because if $\Mo^{\ast \ast} \simeq \M$, then the unit ball of $\Mo$ is weak$^\ast$ dense in the unit ball of $\M$. The second ingredient is an inclusion of $\M$ into a reflexive space; in our case, $\ell^2(\N)$. Analogous theorems could be proven for many other linear spaces of bounded and compact operators using the same technique. 

\section{Results}
For a sequence $a$ and $0 < r < 1$, let 
$$D_r a(n) = r^{\sum_{j=1}^\infty j\kappa_j} a(n), \textrm{ where }  n = \prod_{j=1}^\infty p_j^{\kappa_j}.$$
Note that
$$\sum_{\kappa} r^{2 \sum_{j=1}^\infty j\kappa_j} = \prod_{j=1}^\infty \frac{1}{1-r^{2j}} < \infty.$$
Hence it follows by the dominated convergence theorem that $D_r \colon \ell^2(\mathbb{N}) \to \ell^2(\mathbb{N})$ is a compact operator. Furthermore, $D_r$ is self-adjoint and contractive, $\|D_r\|_{\mathcal{B}(\ell^2(\N))} \leq 1$. The dominated convergence theorem also implies that $D_r \to \Id_{\ell^2(\N)}$  in the strong operator topology (SOT) as $r \to 1$, that is, $\lim_{r \to 1} D_r a = a$ in $\ell^2(\N)$, for every $a \in \ell^2(\N)$. A study of the operators $D_r$ in the context of Hardy spaces of the infinite polytorus can be found in \cite{AOS}.

The Dirichlet convolution of two sequences $a$ and $b$ is the new sequence $a \star b$ given by
$$(a \star b)(n) = \sum_{k|n} a(k) \overline{b(n/k)}, \quad n \in \N.$$
If $a$ and $b$ are two finite sequences, then
\begin{equation}
\label{eq:helsondirichlet}
\langle M(\alpha)a, b \rangle_{\ell^2(\N)} = (\alpha, a \star b),
\end{equation}
where $( a, b ) = \sum_{n=1}^\infty a(n) b(n)$ denotes the bilinear pairing between $a, b \in \ell^2(\N)$. Note also that, for $0 < r < 1$,
\begin{equation}
\label{eq:dilationdirichlet}
D_r(a \star b) = D_ra \star D_r b.
\end{equation}

The following simple lemma is key.
\begin{lem} \label{lem:dilation}
Let $M(\alpha)$ be a bounded multiplicative Hankel operator, $M(\alpha) \in \M$. For $0 < r < 1$, let $\alpha_r = D_r\alpha$. Then $M_{\alpha_r} \in \Mo$, 
$$\|M_{\alpha_r}\|_{\mathcal{B}(\ell^2(\N))} \leq \|M_\alpha\|_{\mathcal{B}(\ell^2(\N))},$$
and $M_{\alpha_r} \to M_\alpha$ and $M_{\alpha_r}^\ast \to M_{\alpha}^\ast$ SOT as $r \to 1$.
\end{lem}
\begin{proof}
By \eqref{eq:helsondirichlet} and \eqref{eq:dilationdirichlet}, it holds for finite sequences $a$ and $b$ that
$$\langle M(\alpha_r)a, b \rangle_{\ell^2(\N)} = \langle M_\alpha D_r a, D_r b \rangle_{\ell^2(\N)}.$$
Hence $M_{\alpha_r} = D_r M_\alpha D_r$. Thus $M_{\alpha_r}$ is compact, $\|M_{\alpha_r}\|_{\mathcal{B}(\ell^2(\N))} \leq \|M_\alpha\|_{\mathcal{B}(\ell^2(\N))}$, and
$M_{\alpha_r} \to M_\alpha$
SOT as $r \to 1$. Similarly, $M_{\alpha_r}^\ast = M_{\overline{\alpha}_r} \to M_{\overline{\alpha}} = M_{\alpha}^\ast$ SOT as $r \to 1$.
\end{proof}
The following is a recognizable consequence, cf. \cite[Theorem 1]{Sara75}. Note that if $S_n$ and $T_n$ are operators such that $S_n \to S$ and $T_n \to T$ SOT, and if $C$ is a compact operator, then $S_n C T_n^{\ast} \to SCT^{\ast}$ in operator norm.
\begin{prop}
	Let $M(\alpha) \in \M$. Then $M(\alpha) \in \Mo$ if and only if 
	\begin{equation} \label{eq:compactconv}
	\lim_{r \to 1}\|M(\alpha_r) - M(\alpha)\|_{\M} = 0.
	\end{equation}
\end{prop}
\begin{proof}
	If \eqref{eq:compactconv} holds, then $M(\alpha) \in \Mo$, since $M(\alpha_r)$ is compact for every $0 < r < 1$. If $M(\alpha) \in \Mo$, then \eqref{eq:compactconv} holds, since $M(\alpha_r) = D_r M(\alpha) D_r = D_r M(\alpha) D_r^\ast$ and $D_r \to \Id_{\ell^2(\N)}$ SOT as $r \to 1$.
\end{proof}

Recall next the main tool from \cite{ABJS79}.
\begin{thm}[\cite{ABJS79}] \label{thm:axler}
Let $T \colon \ell^2(\mathbb{N}) \to \ell^2(\mathbb{N})$ be a non-compact operator and $(T_n)$ a sequence of compact operators such that $T_n \to T$ SOT and $T_n^\ast \to T^\ast$ SOT. Then there exists a sequence $(c_n)$ of non-negative real numbers such that $\sum_n c_n = 1$ for which the compact operator
$$J = \sum_n c_n T_n$$
satisfies
$$\|T - J\|_{\mathcal{B}(\ell^2(\N))} = \inf \left \{\|M(\alpha) - K \|_{\mathcal{B}(\ell^2(\N))} \, : \, K \in \mathcal{K}(\ell^2(\N)) \right\}.$$
\end{thm}
Lemma~\ref{lem:dilation} and Theorem~\ref{thm:axler} immediately yield the existence part of Theorem~\ref{thm:distcpct}.
\begin{proof}[Proof of Theorem~\ref{thm:distcpct}]
Let $M(\alpha)$ be a bounded multiplicative Hankel operator and let $(r_k)$ be a sequence such that $0 < r_k < 1$ and $r_k \to 1$. Then $M(\alpha)$ has a best compact approximant of the form
$$N = \sum_k c_k M(\alpha_{r_k}).$$
But then $N = N(\beta)$ is a multiplicative Hankel operator, where
$\beta = \sum_k c_k \alpha_{r_k}.$

The non-uniqueness of $N(\beta)$ follows immediately once we have established Theorem~\ref{thm:mideal}, by general M-ideal results \cite{HSW75}. In fact, if $M(\alpha) \not \in \Mo$, then the set of minimizers $N(\beta)$ is so large that it spans $\Mo$.
\end{proof}

Note that
$$\|M(\alpha)\|_{\mathcal{B}(\ell^2(\N))} \geq \varlimsup_{N\to\infty} \frac{1}{\|(\alpha(n))_{n=1}^N\|_{\ell^2(\mathbb{N})}}\sum_{n=1}^N |\alpha(n)|^2 = \|\alpha\|_{\ell^2(\mathbb{N})}.$$
Therefore the inclusion $I \colon \Mo \to \ell^2(\mathbb{N})$ is a contractive operator, $Im = I(M(\alpha)) = \alpha$. We can state Theorem~\ref{thm:mideal} slightly more precisely in terms of $I$.
\newtheorem*{thm:mideal}{Theorem~\ref{thm:mideal}}
\begin{thm:mideal}
Consider the bitranspose $U = I^{\ast \ast} \colon \Mo^{\ast \ast} \to \ell^2(\N)$. Then $U \Mo^{\ast \ast} = \M$, viewing $\M$ as a (non-closed) subspace of $\ell^2(\N)$. Furthermore,
$$U \iota_{\Mo} m = m, \quad m \in \Mo,$$
and
$$\|U m^{\ast \ast}\|_{\M} = \|m^{\ast \ast}\|_{\Mo^{\ast \ast}}, \quad m^{\ast \ast} \in \Mo^{\ast \ast}.$$
If $V \colon \Mo^{\ast \ast} \to \M$ is another isometric isomorphism such that $V \iota_{\Mo}m = m$ for all $m \in \Mo$, then $V = U$. Furthermore, $\Mo$ is an M-ideal in $\M$.
\end{thm:mideal}
\begin{proof}
We identify $\left(\ell^2(\N)\right)^\ast \simeq \ell^2(\N)$ linearly through the pairing $( a, b ) = \sum_{n=1}^\infty a(n) b(n)$ between $a, b \in \ell^2(\N)$. 
With this convention, $I^\ast \colon \ell^2(\N) \to \Mo^\ast$ is also contractive, and  
$$I^\ast a(m) = ( \alpha, a ), \quad a \in \ell^2(\N), \; m = M(\alpha) \in \Mo.$$
Since $I$ is injective, $I^\ast$ has dense range. In particular, $\Mo^\ast$ is separable. Furthermore, $I^{\ast \ast} \colon \Mo^{\ast \ast} \to \ell^2(\N)$ is injective. By the reflexivity of $\ell^2(\N)$, we have that $I^{\ast \ast} \iota_{\Mo} = I$, since
$$(I^{\ast \ast} \iota_{\Mo} m, a) = \iota_{\Mo} m( I^\ast a) = (\alpha, a) = (Im, a), \quad m = M(\alpha) \in \Mo, \; a \in \ell^2(\N). $$
The interpretation, viewing $\M$ as a non-closed subspace of $\ell^2(\N)$, is that $I^{\ast \ast} \iota_{\Mo}m = m$, for all $m \in \Mo$.

Consider any $m^{\ast \ast} \in \Mo^{\ast \ast}$, and let $\alpha = I^{\ast \ast}m^{\ast \ast} \in \ell^2(\N)$. Since $\Mo^\ast$ is separable, the weak$^\ast$ topology of the unit ball $B_{\Mo^{\ast \ast}}$ of $\Mo^{\ast \ast}$ is metrizable. As is the case for every Banach space, $\iota_{\Mo}(B_{\Mo})$ is weak$^\ast$ dense in $B_{\Mo^{\ast \ast}}$. Hence there is a sequence $(m_n)_{n=1}^\infty$ in $\Mo$ such that $\iota_{\Mo} m_n \to m^{\ast \ast}$ weak$^\ast$ and $\|m_n\|_{\mathcal{B}(\ell^2(\N))} \leq \|m^{\ast \ast}\|_{\Mo^{\ast \ast}}.$ Suppose that $m_n = M(\alpha_n)$ and let $a, b \in \ell^2(\N)$ be two finite sequences. Then, since $\iota_{\Mo} m_n \to m^{\ast \ast}$  weak$^\ast$,
$$\langle M(\alpha_n) a, b \rangle_{\ell^2(\N)} = (\alpha_n, a \star b) = I^\ast (a\star b)(m_n)  \to m^{\ast \ast}(I^\ast (a\star b)) = (\alpha, a\star b),$$
as $n \to \infty$.
It follows that
$$|\langle M(\alpha) a, b \rangle_{\ell^2(\N)}| = |(\alpha, a \star b)| \leq \varlimsup_{n \to \infty} \|m_n\|_{\mathcal{B}(\ell^2(\N))} \|a\|_{\ell^2(\N)} \|b\|_{\ell^2(\N)} \leq \|m^{\ast \ast}\|_{\Mo^{\ast \ast}} \|a\|_{\ell^2(\N)} \|b\|_{\ell^2(\N)}.$$
Since $a,b$ were arbitrary finite sequences, it follows that $M(\alpha) \in \M$ and 
$$\|M(\alpha)\|_{\mathcal{B}(\ell^2(\N))} \leq \|m^{\ast \ast}\|_{\Mo^{\ast \ast}}.$$ Since $\alpha = I^{\ast \ast}m^{\ast \ast}$ this proves that $I^{\ast \ast}$ maps $\Mo^{\ast \ast}$ contractively into $\M$.

Conversely, suppose that $m = M(\alpha) \in \M$. By Lemma~\ref{lem:dilation}, for $0 < r < 1$, $M(\alpha_r) \in \Mo$, $\|M(\alpha_r)\| \leq \|M(\alpha)\|$, and $\alpha_r \to \alpha$ in $\ell^2(\N)$ as $r \to 1$. Define $m^{\ast \ast} \in \Mo^{\ast \ast}$ by
\begin{equation} \label{eq:mstarstardef}
m^{\ast \ast} (I^\ast a) := (\alpha, a) = \lim_{r \to 1}  (\alpha_r, a) = \lim_{r \to 1}  I^{\ast}a(M(\alpha_r)),\quad a \in \ell^2(\N). 
\end{equation}
This specifies an element $m^{\ast \ast} \in \Mo^{\ast \ast}$ since $I^\ast$ has dense range in $\Mo^\ast$ and
$$ |m^{\ast \ast} (I^\ast a)| \leq \varlimsup_{r \to 1} \|M(\alpha_r)\|_{\mathcal{B}(\ell^2(\N))} \|I^\ast a\|_{\Mo^\ast} \leq  \|M(\alpha)\|_{\mathcal{B}(\ell^2(\N))} \|I^\ast a\|_{\Mo^\ast}.$$
From this inequality we also see that 
\begin{equation} \label{eq:expansive}
\|m^{\ast \ast}\|_{\Mo^{\ast \ast}} \leq \|m\|_{\mathcal{B}(\ell^2(\N))}.
\end{equation} Furthermore, since
$$(I^{\ast \ast}m^{\ast \ast}, a) = m^{\ast \ast}(I^\ast a) = (\alpha, a), \quad a \in \ell^2(\N),$$
we have that $I^{\ast \ast}m^{\ast \ast} = \alpha$. Hence $I^{\ast \ast}$ maps $\Mo^{\ast \ast}$ bijectively and contractively onto $\M$. By \eqref{eq:expansive}, $I^{\ast \ast} \colon \Mo^{\ast \ast} \to \M$ is also expansive, and hence it is an isometric isomorphism.

Recall that $\mathcal{K}(\ell^2(\N))$ is an M-ideal in $\mathcal{B}(\ell^2(\N))$ \cite{Dix50} -- indeed, $\mathcal{K}(\ell^2(\N))$ is a two-sided closed ideal in $\mathcal{B}(\ell^2(\N))$. It is well known that there is an isometric isomorphism $E \colon \mathcal{K}(\ell^2(\N))^{\ast \ast} \to \mathcal{B}(\ell^2(\N))$ such that $E\iota_{\mathcal{K}(\ell^2(\N))} K = K$ for all $K \in \mathcal{K}(\ell^2(\N))$. Thus $\mathcal{K}(\ell^2(\N))$ is M-embedded. Since $\Mo$ is a closed subspace of $\mathcal{K}(\ell^2(\N))$, $\Mo$ is also M-embedded \cite[Theorem III.1.6]{HWW}. Hence, since we have shown that $I^{\ast \ast} \colon \Mo^{\ast \ast} \to \M$ is an isometric isomorphism for which $I^{\ast \ast}\iota_{\Mo}m = m$ for all $m \in \Mo$, it follows that $\Mo$ is an M-ideal in $\M$.

Finally, if $V \colon \Mo^{\ast \ast} \to \M$ is another isometric isomorphism such that $V\iota_{\Mo}m = m$, $m \in \Mo$, then $F = V^{-1}I^{\ast \ast} \colon \Mo^{\ast \ast} \to \Mo^{\ast \ast}$ is an isometric isomorphism such that $F\iota_{\Mo} = \iota_{\Mo}$. However, since $\Mo$ is M-embedded, $F$ must be obtained as the bitranspose, $F = G^{\ast \ast}$, of an isometric isomorphism $G \colon \Mo \to \Mo$ \cite[Proposition III.2.2]{HWW}. But then $G = \Id_{\Mo}$, since
$$m^\ast(Gm) = G^\ast m^\ast (m) = F \iota_{\Mo} m (m^\ast) = m^\ast(m), \quad m \in \Mo, \; m^\ast \in \Mo^\ast.$$
Hence $F = \Id_{\Mo^{\ast \ast}}$ and so $V = I^{\ast \ast}$.
\end{proof}

The predual of a space of Hankel operators usually has an abstract description as a projective tensor product \cite{ARSW10, CRW76, FL02}. In the present context, let
$$X = \left\{c \,:\, c = \sum_{\textrm{finite}} a_k \star b_k, \; a_k, b_k \textrm{ finite sequences} \right\},$$
and equip $X$ with the norm
$$\|c\|_{X} = \inf \sum_{\textrm{finite}} \|a_k\|_{\ell^2(\N)} \|b_k\|_{\ell^2(\N)},$$
where the infimum is taken over all \textbf{finite} representations of $c$. By writing $c = c \star (1,0,0, \ldots)$ it is clear that $\|c\|_{X} \leq \|c\|_{\ell^2(\N)}$ for $c \in X$.  

We define the projective tensor product space $\mathcal{X} = \ell^2(\N) \; \hat{\star} \; \ell^2(\N)$ with respect to Dirichlet convolution as the Banach space completion of $X$. It is essentially definition that $\mathcal{X}^\ast \simeq \M$.
\begin{lem}\label{lem:predual}
For $m = M(\alpha) \in \M$, let 
$$Jm(c) = (\alpha, c), \quad c \in X.$$ 
Then $Jm$ extends to a bounded functional on $\mathcal{X}$ for every $m \in \M$, and $J \colon \M \to \mathcal{X}^\ast$ is an isometric isomorphism.
\end{lem} 
\begin{proof}
Let $m \in \M$. If $c \in X$ and $\varepsilon > 0$, choose a representation $c = \sum_{k=1}^N a_k \star b_k$, where $a_k$ and $b_k$ are finite sequences for every $k$, and $\sum_{k=1}^N \|a_k\|_{\ell^2(\N)} \|b_k\|_{\ell^2(\N)} < \|c\|_X + \varepsilon$. Then
$$\left|Jm(c)\right| = \left|\sum_{k=1}^N \langle M(\alpha)a_k, b_k \rangle_{\ell^2(\N)}\right| \leq \|m\|_{\mathcal{B}(\ell^2(\N))}(\|c\|_{X} + \varepsilon).$$
Hence $\|Jm\|_{\mathcal{X}^\ast} \leq \|m\|_{\mathcal{B}(\ell^2(\N))}$. Choosing finite sequences $a$ and $b$ such that $\|a\|_{\ell^2(\N)} =\|b\|_{\ell^2(\N)} = 1$ and $\langle M(\alpha)a, b \rangle_{\ell^2(\N)} > \|m\|_{\mathcal{B}(\ell^2(\N))} - \varepsilon$, and letting $c = a \star b$ gives that
$$\|m\|_{\mathcal{B}(\ell^2(\N))} - \varepsilon < \|Jm\|_{\mathcal{X}^\ast}\|c\|_X \leq \|Jm\|_{\mathcal{X}^\ast}.$$
 Hence $J$ is an isometry. 

The inclusion of finite sequences into $X$ extends to a contraction $E \colon \ell^2(\N) \to \mathcal{X}$. Let $\ell \in \mathcal{X}^\ast$ and let $c \in X$. Then $\ell(c) = (\alpha, c)$, where $\alpha = E^\ast \ell \in \ell^2(\N)$. Then $m = M(\alpha) \in \M$, since $\ell \in \mathcal{X}^\ast$. Clearly $Jm = \ell$ and thus $J$ is onto.
\end{proof}

\newtheorem*{thm:midspace}{Theorem~\ref{thm:midspace}}
\begin{thm:midspace} 
For every $c \in X$, let
$$Lc(m) = (\alpha, c), \quad m = M(\alpha) \in \Mo.$$
Then $L$ extends to an isometric isomorphism $L \colon \mathcal{X} \to \Mo^\ast$, and
$$L^\ast U^{-1} = J \colon \M \to \mathcal{X}^\ast$$
is the isometric isomorphism of Lemma~\ref{lem:predual}. Here $U \colon \Mo^{\ast \ast} \to \M$ is the isometric isomorphism of Theorem~\ref{thm:mideal}.
\end{thm:midspace}
\begin{proof}
The quickest proof proceeds by noting that $\Mo^{\ast}$ is a strongly unique predual of $\Mo^{\ast \ast}$, since $\Mo$ is M-embedded. This implies that the isometric isomorphism $JU \colon \Mo^{\ast \ast} \to \mathcal{X}^\ast$ is the adjoint of an isometric isomorphism $E \colon \mathcal{X} \to \Mo^\ast$, $E^\ast = JU$. But then, for $c \in X$ and $m = M(\alpha) \in \Mo$,
\begin{equation} \label{eq:Lcomp}
Ec(m) = \iota_{\Mo}m(Ec) = E^\ast \iota_{\Mo}m (c) = JU \iota_{\Mo}m (c) = Jm(c) = (\alpha, c) = Lc(m).
\end{equation}
Hence $L = E$, and thus $L$ is an isometric isomorphism.

Alternatively, the weak$^\ast$-weak$^\ast$ continuity of $JU$ can be proven by hands. $L$ clearly extends to a contractive operator $L \colon \mathcal{X} \to \Mo^\ast$. The computation \eqref{eq:Lcomp} shows that $JU \iota_{\Mo} = L^\ast \iota_{\Mo}$. Let $m^{\ast \ast} \in \Mo^{\ast \ast}$ and let $M(\alpha) = Um^{\ast \ast}$. From \eqref{eq:mstarstardef} we deduce that $m^{\ast \ast}_r = \iota_{\Mo} M(\alpha_r) \to m^{\ast \ast}$ weak$^\ast$ in $\Mo^{\ast \ast}$. Hence $L^{\ast} m^{\ast \ast}_r \to L^{\ast} m^{\ast \ast}$ weak$^\ast$ in $\mathcal{X}^\ast$. On the other hand, for $c \in X$,
$$JUm^{\ast \ast}(c) = (\alpha, c) = \lim_{r \to 1} (\alpha_r, c) = \lim_{r \to 1} JUm^{\ast \ast}_r(c) = \lim_{r \to 1} L^{\ast} m^{\ast \ast}_r(c) = L^{\ast} m^{\ast \ast}(c).$$
This shows that $JU = L^{\ast}$, and hence $L$ is an isometric isomorphism.
\end{proof}
\begin{remark}
In the notation of Theorem~\ref{thm:mideal}, $I^\ast c = L c$ for $c \in X$. Theorem~\ref{thm:midspace} hence completes the picture of Theorem~\ref{thm:mideal} by giving an interpretation of the operator $I^\ast$.
\end{remark}
 Suppose that we had instead defined the projective tensor product space $ \ell^2(\N) \; \hat{\star} \; \ell^2(\N)$ as the sequence space
$$\mathcal{Y} = \left\{c \,:\, c = \sum_{k=1}^\infty a_k \star b_k, \; a_k, b_k \in \ell^2(\N), \; \sum_{k=1}^\infty \|a_k\|_{\ell^2(\N)} \|b_k\|_{\ell^2(\N)} < \infty \right\},$$
normed by
$$\|c\|_{\mathcal{Y}} = \inf \sum_{k=1}^\infty \|a_k\|_{\ell^2(\N)} \|b_k\|_{\ell^2(\N)},$$
where the infimum is taken over all representations of $c$. One would like to know that $\mathcal{Y} = \mathcal{X}$. Indeed, it is not a priori clear that $\mathcal{X}$ is a sequence space; or if $\mathcal{X}$ is identifiable with a space of Dirichlet series, if considering multiplicative Hankel operators in that context. For $\mathcal{Y}$ these properties are immediate.
\begin{lem} \label{lem:complete}
$\mathcal{Y}$ is a Banach space.
\end{lem}
\begin{proof}
Since $|(a \star b)(n)| \leq \|a\|_{\ell^2(\N)}\|b\|_{\ell^2(\N)}$ it is clear that 
$$e_n(c) = c(n), \quad c \in \mathcal{Y},$$
defines an element $e_n \in \mathcal{Y}^\ast$, for every $n \in \N$. It follows that $\|c\|_{\mathcal{Y}} = 0$ if and only if $c = 0$.

Suppose that $\sum_{k=1}^\infty c_k$ is an absolutely convergent series in $\mathcal{Y}$. Then there are double sequences $(a_{k,j})$ and $(b_{k,j})$ such that
$c_k = \sum_{j=1}^\infty a_{k,j} \star b_{k,j}$ for every $k$ and
$$\sum_{k,j=1}^\infty \|a_{k,j}\|_{\ell^2(\N)}\|b_{k,j}\|_{\ell^2(\N)} < \infty.$$
Then $c = \sum_{k,j=1}^\infty a_{k,j} b_{k,j}$ is an element of $\mathcal{Y}$ and 
$$\|c - \sum_{k=1}^N c_k\|_{\mathcal{Y}} \leq \sum_{k=N+1}^\infty \sum_{j=1}^\infty \|a_{k,j}\|_{\ell^2(\N)}\|b_{k,j}\|_{\ell^2(\N)} \to 0, \quad N \to \infty.$$
Hence $\sum_{k=1}^\infty c_k$ converges in $\mathcal{Y}$ to $c$. Thus $\mathcal{Y}$ is complete.
\end{proof}
We now prove that $\mathcal{Y} = \mathcal{X}$. The details are similar to those of \cite{RS14}, where projective tensor products of spaces of holomorphic functions were considered. Note that $X$ is contractively contained in $\mathcal{Y}$.
\begin{prop} \label{prop:altdef}
The inclusion $V \colon X \to \mathcal{Y}$ extends to an isometric isomorphism $V \colon \mathcal{X} \to \mathcal{Y}$.
\end{prop}
\begin{proof}
We make the following preliminary observation. Since for every $0 < r < 1$,
$$D_r( a \star b) = D_r a \star D_r b, \quad \|D_r\|_{\mathcal{B}(\ell^2(\N))} \leq 1,$$ $D_r$ defines a bounded operator $D_r \colon \mathcal{X} \to \mathcal{X}$,
$$\|D_r\|_{B(\mathcal{X})} \leq 1.$$
Furthermore, since $D_r \to \Id_{\ell^2(\N)}$ SOT on $\ell^2(\N)$ as $r \to 1$, it follows that $\|D_r c - c\|_{X} \leq \|D_r c - c\|_{\ell^2(\N)} \to 0$ as $r \to 1$ for every $c \in X$. Hence $D_r \to \Id_{\mathcal{X}}$ SOT on $\mathcal{X}$ as $r \to 1$.

As in Lemma~\ref{lem:complete}, for each $n \in \N$, 
$$e_n(c) = c(n), \quad c \in X,$$
extends to a functional $e_n \in \mathcal{X}^\ast$ with $\|e_n\|_{\mathcal{X}^\ast} \leq 1$. We show now that $(e_n)$ is a complete sequence in $\mathcal{X}^\ast$ with respect to the weak$^\ast$ topology. Suppose that $c \in \mathcal{X}$ and that $e_n(c) = 0$ for all $n$. Pick a sequence $(c_k)$ in $X$ such that $c_k \to c$ in $\mathcal{X}$. Then for fixed $r < 1$,
$$\|D_r c\|_{\mathcal{X}} \leq \varlimsup_{k \to \infty}\left(\|D_r(c - c_k)\|_{\mathcal{X}} + \|D_r c_k\|_{\mathcal{X}} \right) = \varlimsup_{k \to \infty}\|D_r c_k\|_{X} \leq \varlimsup_{k \to \infty}\|D_r c_k\|_{\ell^2(\N)}. $$
Since $c_k \to c$ in $\mathcal{X}$ and $e_n \in \mathcal{X}^\ast$, we have that $\lim_{k\to\infty}c_k(n) = e_n(c) = 0$ for every $n$.
Furthermore, $|c_k(n)| \leq \|e_n\|_{\mathcal{X}^\ast} \|c_k\|_{X} \leq \|c_k\|_{X}$ is uniformly bounded in $k$ and $n$. Hence it follows by the dominated convergence theorem that $\varlimsup_{k \to \infty}\|D_r c_k\|_{\ell^2(\N)} = 0$ and thus that $D_r c = 0$. Since $D_r c \to c$ in $\mathcal{X}$ as $r \to 1$ we conclude that $c = 0$. Therefore $(e_n)$ is complete.

Hence $\mathcal{X}$ is a space of sequences. More precisely, since every evaluation $e_n$ is a bounded functional on $\mathcal{Y}$ as well, the extension $V \colon \mathcal{X} \to \mathcal{Y}$ of the inclusion map is given by
\begin{equation} \label{eq:Jformula}
Vc = (e_n(c))_{n=1}^\infty, \quad c \in \mathcal{X}.
\end{equation}
The completeness of $(e_n)$ implies that $V$ is injective. 

We next prove that $V$ is onto. The argument is precisely as in \cite{RS14}, but we include it for completeness. For a sequence $a$ and $m \in \N$, let $a^m = (a(1), \ldots, a(m), 0, \ldots).$ Given $a \in \ell^2(\N)$ and $\delta > 0$, choose a sequence $(m_1, m_2, \ldots)$ such that $\|a - a^{m_k}\|_{\ell^2(\N)} \leq 2^{-k}$. Let $a_k = a^{m_{k+1}} - a^{m_k}$. Then, for sufficiently large $K$,
 $$a = a^{m_K} + \sum_{k=K}^\infty (a^{m_{k+1}} - a^{m_k}), \quad \sum_{k=K}^\infty \|a^{m_{k+1}} - a^{m_k}\|_{\ell^2(\N)} <  \delta.$$
 Hence we can write $a = \sum_{j=1}^\infty a_j$, where each $a_j$ is a finite sequence and $\sum_j \|a_j\|_{\ell^2(\N)} < \|a\|_{\ell^2(\N)} + \delta$.
 
 Given $c \in \mathcal{Y}$ and $\varepsilon > 0$,
  choose $(a_{k})_{k=1}^\infty$ and $(b_{k})_{k=1}^\infty$ such that 
 $$c = \sum_{k=1}^\infty a_{k} \star b_{k}, \quad \sum_{k=1}^\infty \|a_{k}\|_{\ell^2(\N)} \|b_{k}\|_{\ell^2(\N)} < \|c\|_{\mathcal{Y}} + \varepsilon.$$
 For each $k$, write, as in the preceding paragraph, $a_k = \sum_{j=1}^\infty a_{k,j}$, $b_k = \sum_{j=1}^\infty b_{k,j}$, where each $a_{k,j}$ and $b_{k,j}$ is a finite sequence and 
 $$\sum_{j=1}^\infty \|a_{k,j}\|_{\ell^2(\N)} < \|a_k\|_{\ell^2(\N)} + \delta_k, \quad \sum_{j=1}^\infty \|b_{k,j}\|_{\ell^2(\N)} < \|b_k\|_{\ell^2(\N)} + \delta_k.$$
 Here the $\delta_k$ are chosen so that
 $$\sum_{k=1}^\infty (\|a_k\|_{\ell^2(\N)} + \delta_k)(\|b_k\|_{\ell^2(\N)} + \delta_k) < \sum_{k=1}^\infty \|a_{k}\|_{\ell^2(\N)} \|b_{k}\|_{\ell^2(\N)} + \epsilon.$$
 Then 
 $$c = \sum_{k, j, l=1}^\infty a_{k,j} \star b_{k,l}, \quad \sum_{k, j, l=1}^\infty \|a_{k,j}\|_{\ell^2(\N)}  \|b_{k,l}\|_{\ell^2(\N)} < \sum_{k=1}^\infty \|a_{k}\|_{\ell^2(\N)} \|b_{k}\|_{\ell^2(\N)} + \epsilon < \|c\|_{\mathcal{Y}} + 2\varepsilon.$$
 
 Relabeling, we have a representation $c = \sum_{n=1}^\infty a_n \star b_n$ where $a_n$ and $b_n$ are finite sequences and $\sum_n \|a_n\|_{\ell^2(\N)}\|b_n\|_{\ell^2(\N)} < \|c\|_{\mathcal{Y}} + 2\varepsilon$. Let $c_N = \sum_{n=1}^N a_n \star b_n$. Then $c_N \to c$ in $\mathcal{Y}$, and furthermore $(c_N)$ is a Cauchy sequence in $X$, hence has a limit $\tilde{c}$ in $\mathcal{X}$. By continuity of the functionals $e_n$ on both $\mathcal{Y}$ and $\mathcal{X}$, we find in view of \eqref{eq:Jformula} that $V\tilde{c} = c$. Hence $V$ is onto.
 
  Furthermore, since $V$ is contractive,
 $$\|c\|_{\mathcal{Y}} \leq \|\tilde{c}\|_{\mathcal{X}} = \lim_{N\to\infty} \|c_N\|_{X} < \|c\|_{\mathcal{Y}} + 2\varepsilon. $$
 We already showed that $V$ is injective, so that $\tilde{c}$ is uniquely defined by $c$.  On the other hand, $\varepsilon$ is arbitrary. We conclude that $\|c\|_{\mathcal{Y}} = \|\tilde{c}\|_{\mathcal{X}}$. It follows that $V$ is an isometric isomorphism.
\end{proof}

\bibliographystyle{amsplain}
\bibliography{cpctbddhelson}

\end{document}